\documentclass[multi]{cambridge7A}
\usepackage[UKenglish]{babel}
\usepackage[sectionbib]{natbib}
\usepackage{chapterbib}
\usepackage{amsmath,amssymb,amsthm,amsfonts}
\usepackage{enumerate}

\setcitestyle{numbers,square,comma}

\theoremstyle{plain}
\newtheorem{theorem}{Theorem}[section]
\newtheorem*{Cor}{Corollary}
\newtheorem*{KBD}{Theorem KBD (Kestelman--Borwein--Ditor Thorem, KBD)}
\newtheorem*{Baire}{Theorem B (Baire's Theorem---almost completeness of
  Baire sets)}
\newtheorem*{Lusin}{Theorem L (Lusin's Theorem; {\rm cf.\ \cite[end of
  Section 55]{Hal}})}
\newtheorem*{MS}{Lemma S (Multiplicative Sierpi\'nski Lemma;
  {\rm \cite{Sier}})}
\newtheorem*{Sier}{Theorem S {\rm (Sierpi\'nski, \cite{Sier})}}
\newtheorem*{KCat}{Theorem \ref{t:6}C (Kingman Theorem for Category)}
\newtheorem*{KMeas}{Theorem \ref{t:6}M (Kingman Theorem for Measure)}
\newtheorem*{Kemp}{Theorem K (Displacements Lemma---Kemperman's Theorem;
  {\rm \cite[Th.\ 2.1]{Kem} with $B_{i}=E$, $a_{i}=t$})}
\newtheorem*{Ruz}{Theorem R (Ruziewicz's Theorem {\rm \cite{Ruz}};
  {\rm cf.\ \cite{Kem} after Lemma 2.1 for the measure case})}
\newtheorem*{BHW}{Theorem BHW {\rm \cite[Th.\ 2.6 and 2.7]{BHW}}}

\theoremstyle{definition}
\newtheorem*{defn}{Definition}
\newtheorem*{exas}{Examples}
\newtheorem*{DN}{Definitions and notation (Essential contiguity conditions)}
\newtheorem*{WA}{Definition (Weakly Archimedean property---an admissibility
condition)}

\newtheorem*{rem}{Remark}
\newtheorem*{rems}{Remarks}
\newtheorem*{nota}{Notation}

\allowdisplaybreaks[1]

\newcommand\norm[1]{\lVert #1\rVert }
\def\cl{\mathop{\rm cl}}
\def\esssup{\mathop{\rm ess\,sup}}

\providecommand*\Index[1]{#1\index{#1}}
\providecommand*\undex[1]{} 

\begin{document}
\alphafootnotes
\author[N. H. Bingham and A. J. Ostaszewski]{N. H. Bingham\footnotemark[1]
  and A. J. Ostaszewski\footnotemark[2]\,\footnotemark[3]}
\chapter{Kingman, category and combinatorics}
\footnotetext[1]{Mathematics Department, Imperial College London, London SW7
  2AZ; n.bingham@ic.ac.uk, nick.bingham@btinternet.com}
\footnotetext[2]{Mathematics Department, London School of Economics, Houghton
  Street, London WC2A 2AE; a.j.ostaszewski@lse.ac.uk}
\footnotetext[3]{This paper was completed on the final day of the life of
  Zofia Ostaszewska, the second author's mother, and is dedicated to her
  memory.}
\arabicfootnotes
\contributor{Nicholas H. Bingham
  \affiliation{Imperial College London}}
\contributor{Adam J. Ostaszewski
  \affiliation{London School of Economics}}
\renewcommand\thesection{\arabic{section}}

\begin{abstract}Kingman's Theorem on skeleton limits---passing from
limits as $n\rightarrow \infty $ along $nh$ ($n\in \mathbb{N}$) for enough
$h>0$ to limits as $t\rightarrow \infty $ for $t\in \mathbb{R}$---is
generalized to a Baire/measurable setting via a topological approach. We
explore its affinity with a combinatorial theorem due to Kestelman and to
Borwein and Ditor, and another due to Bergelson, Hindman and Weiss. As
applications, a theory of `rational' skeletons akin to Kingman's integer
skeletons, and more appropriate to a measurable setting, is developed, and
two combinatorial results in the spirit of van der Waerden's celebrated
theorem on arithmetic progressions are given.\end{abstract}

\subparagraph{Keywords}Baire property, bitopology, complete metrizability,
density topology, discrete skeleton, essential contiguity, generic property,
infinite combinatorics, measurable function, Ramsey theory, Steinhaus theory
\subparagraph{AMS subject classification (MSC2010)}26A03

\section{Introduction}
The background to the theme of the title is Feller's theory of
\textit{recurrent events}\index{recurrent event|(}. This goes back to
Feller\index{Feller, W.|(} in 1949 \cite{F1}, and received its first textbook
synthesis in \cite{F2} (see e.g.\ \cite{GS} for a recent treatment). One is
interested in something (`it', let us say for now---we can proceed informally
here, referring to the above for details) that happens (by default, or by
fiat) at time 0, may or may not happen at discrete times $n=1$, 2, \ldots, and
is such that its happening `resets the clock', so that if one treats this
random time as a new time-origin, the subsequent history is a probabilistic
replica of the original situation.  Motivating examples include return to the
origin in a simple random walk (or \index{coin tossing}coin-tossing game);
attaining a new maximum in a simple
\index{random walk (RW)}random walk; returning to the
initial state $i$ in a (discrete-time) Markov chain\index{Markov, A. A.!Markov
chain}.  Writing $u_{n}$ for the probability that `it' happens at time $n$ (so
$u_{0}=1$), one calls $u=(u_{n})$ a \textit{renewal sequence}\index{renewal
theory!renewal sequence}. Writing $f_{n}$ for the probability that `it'
happens \textit{for the first time} at $n>1$ ($f_{0}:=0$), $f=(f_{n})$, the
generating functions\index{probability generating function (pgf)} $U$, $F$ of $u$,
$f$ satisfy the \textit{Feller relation} $U(s)=1/(1-F(s))$.

It is always worth a moment when teaching stochastic processes to ask the
class whether time is discrete or continuous. It is both, but which aspect is
uppermost depends on how we measure, or experience, time---whether our watch
is digital or has a sweep second hand, one might say. In continuous time, one
encounters analogues of Feller's theory above in various probabilistic
contexts---e.g., the server in an $M/G/1$ \Index{queue} being idle. In the
early 1960s, the Feller theory, queueing theory (in the phase triggered by
Kendall's\index{Kendall, D. G.|(} work, \cite{Ken1}) and John
Kingman\index{Kingman, J. F. C.} were all young.  Kingman found himself drawn
to the task of creating a continuous-time version of Feller's theory of
recurrent events (see \cite{King4} for his reminiscences of this time)---a
task he triumphantly accomplished in his theory of \textit{regenerative
phenomena}\index{regenerative phenomenon}, for which his book \cite{King3}
remains the standard source. Here the role of the renewal
sequence\index{renewal theory!renewal sequence} is played by the Kingman
$p$-function\index{pfunction@p-function|(}, where $p(t)$ is the probability that the
regenerative phenomenon $\Phi $ occurs at time $t\geq 0$.

A continuous-time theory contains within itself infinitely many versions of a
discrete-time theory. For each fixed $h>0$, one obtains from a Kingman
regenerative phenomenon $\Phi $ with $p$-function $p(t)$ a
Feller\index{Feller, W.|)} recurrent event\index{recurrent event|)} (or
regenerative phenomenon in discrete time, as one would say nowadays), ${\Phi
}_{h}$ say, with renewal sequence $u_{n}(h)=p(nh)$---called the
discrete \textit{skeleton}\index{skeleton!discrete skeleton} of $\Phi $ for time-step
$h$---the $h$-skeleton, say.

While one can pass from continuous to discrete time by taking skeletons, it is
less clear how to proceed in the opposite direction---how to combine
discrete-time information for various time-steps $h$ to obtain continuous-time
information. A wealth of information was available in the discrete-time
case---for example, limit theorems for Markov chain
transition probabilities. It was tempting to seek to use such
information to study corresponding questions in continuous time, as was done
in \cite{King1}. There, Kingman\index{Kingman, J. F. C.} made novel use of the
Baire category theorem\index{Baire, R.-L.!Baire category theorem}, to extend a
result of Croft\index{Croft, H. T.} \cite{Cro}, making use of a lemma
attributed both to Golomb\index{Golomb, M.} and Gould\index{Gould, S. H.} and
to Anderson\index{Anderson, R. D.} and Fine\index{Fine, N. J.} (see \cite{NW}
for both).

While in the above we have limits at infinity through integer multiples $nx$,
we shall also be concerned with limits through positive rational multiples
$qx$ (we shall always use the notations $\lim_{n}$ and $\lim_{q}$ for
these). There are at least three settings in which such rational limits are
probabilistically relevant:

\begin{enumerate}[(i)]
\item \textit{Infinitely divisible $p$-functions}. The Kingman
$p$-functions\index{pfunction@p-function|)} form a \Index{semigroup} under pointwise
  multiplication (if $p_{i}$ come from $\Phi_{i}$ with $\Phi_{1}$, $\Phi_{2}$
  independent, $p:=p_{1}p_{2}$ comes from $\Phi :=\Phi_{1}\cap \Phi_{2}$, in
  an obvious notation). The arithmetic of this semigroup has been studied in
  detail by Kendall \cite{Ken2}.

\item \textit{Embeddability of infinitely divisible laws}. If a probability
law is infinitely divisible\index{infinite divisibility}, one can define its
$q$th convolution power for any positive rational $q$. The question then
arises as to whether one can embed these rational powers into a continuous
semigroup of real powers.  This is the question
of \textit{embeddability}\index{embeddability}, studied at length (see e.g.\
\cite[Ch.\ III]{Hey})\index{Heyer, H.} in connection with the
L\'{e}vy--Khintchine\index{Levy, P.@L\'evy, P.!L\'evy--Khintchine formula} formula
on locally compact groups\index{group!locally compact group}.

\item \textit{Embeddability of branching processes}\index{branching process}. While for simple
branching processes both space (individuals) and time (generations) are
discrete, it makes sense in considering e.g.\ the \Index{biomass} of large
populations to work with branching processes where space and/or time may be
continuous. While one usually goes from the discrete to the continuous setting
by taking limits, embedding is sometimes possible; see e.g.\ Karlin and
McGregor\index{Karlin, S.}\index{McGregor, J. L.} \cite{KMcG},
Bingham \cite{Bin}.
\end{enumerate}

In addition, (i) led Kendall\index{Kendall, D. G.|)} \cite[Th.\ 16]{Ken2} to
study \textit{sequential regular variation} (see e.g.\ \cite[\S1.9]{BGT}). The
interplay between the continuous and sequential aspects of \Index{regular
variation}, and between measurable and Baire aspects, led us to our recent
theory of \textit{topological regular variation} (see e.g.\ \cite{BOst2} and
our other recent papers), our motivation here.

In Section 2 we discuss the relation between Kingman's Theorem and the
Kestelman--Borwein--Ditor Theorem\index{Kestelman, H.|(}\index{Borwein,
D.|(}\index{Ditor, S. Z.|(} (KBD), introducing definitions and summarizing
background results which we need (including the density
topology)\index{topology!density topology}. In Section 3 we generalize to a
Baire/measurable setting the Kingman Theorem (originally stated for open
sets). Our\break (bi-)topological approach\index{topology!bitopology|(}
(borrowed from \cite{BOst7}) allows the two cases to be treated as one that,
by specialization, yields either case. (This is facilitated by the density
topology.) The theorem is applied in Section 4 to establish a theory of
`rational' skeletons parallel to Kingman's\index{Kingman, J. F. C.} integer
skeletons\index{skeleton!rational skeleton}\index{skeleton!integer skeleton}. In Section 5 we
offer a new proof of KBD in a `consecutive' format suited to proving in
Section 6 combinatorial results in the spirit of van der
Waerden's\index{Waerden, B. L. van der} celebrated theorem on arithmetic
progressions\index{arithmetic progression}. Again a bitopological (actually
`bi-metric') approach allows unification of the Baire/measurable cases. Our
work in Section 6 is based on a close reading of \cite{BHW}\index{Bergelson,
V.}\index{Hindman, N.|(}\index{Weiss, B.}, our debt to which is clear.

\section{Preliminaries}
In this section we motivate and define notions of contiguity\index{contiguity|(}. Then
we gather classical results from topology and measure
theory\index{measure theory|(} (complete
metrizability\index{metrizability!complete metrizability} and the
density topology\index{topology!density topology|(}). We begin by recalling the
following result, in which the expression `for \Index{generically all} $t$'
means for all $t$ except in a meagre\index{meagre|(} or null set according
to context. We use the terms Baire
set/function\index{Baire, R.-L.!Baire set|(}\index{Baire, R.-L.!Baire function}\index{Baire, R.-L.!Baire property}
to mean a set/function with the Baire property. Evidently, the interesting
cases are with $T$ Baire non-meagre/measurable non-null. The result in this
form is due in the measure case to Borwein and Ditor \cite{BoDi}, but was
already known much earlier albeit in somewhat weaker form by
Kestelman\index{Kestelman, H.|)}\index{Borwein, D.|)}\index{Ditor, S. Z.|)}
\cite[Th.\ 3]{Kes}, and rediscovered by Kemperman\index{Kemperman, J. H. B.} \cite{Kem} and later by
Trautner\index{Trautner, R.} \cite{Trau} (see \cite[p.\ xix and footnote p.\
10]{BGT})\index{Goldie, C. M.|(}\index{Teugels, J. L.|(}. We note a cognate result in \cite[Th.\
2.3.7]{HJ}\index{Hoffmann-Jorgensen, J.@Hoffmann-J\o rgensen, J.}.

\begin{KBD}\hskip0pt minus10mm Let\break $\{z_{n}\}\rightarrow 0$ be a
null sequence of reals. If $T$
is Baire/Lebesgue-measurable\undex{Lebesgue, H. L.!Lebesgue measurable}, then for generically
all $t\in T$ there is an infinite set $\mathbb{M}_{t}$ such that
\[
\{t+z_{m}:m\in \mathbb{M}_{t}\}\subseteq T.
\]
\end{KBD}

We give a new unified proof of the measure and Baire cases in Section 5, based
on `almost-complete
metrizability\index{metrizability!almost-complete metrizability}'
(= \textit{almost complete} + \textit{complete metrizability}, see below) and
a \textit{Generic Dichotomy} given in Section 3; earlier unification was
achieved through a bitopological approach (as here to the
Kingman\index{Kingman, J. F. C.} Theorem) in \cite{BOst7}. This result is a
theorem about additive infinite combinatorics\index{additive
combinatorics}. It is of fundamental and unifying importance in contexts where
additive structure is key; its varied applications include proofs of classical
results such as Ostrowski's\index{Ostrowski, A. M.} Theorem on the continuity
of Baire/Lebesgue convex\index{convexity} (and so additive)
functions\index{additive function} (cf.\ \cite{BOst8}), a plethora of results
in the theory of \index{subadditivity}subadditive functions
(cf.\ \cite{BOst1,BOst4}), the Steinhaus\index{Steinhaus, H. D.} Theorem on
distances
\cite{BOst8,BOst6} and the Uniform Convergence Theorem of \Index{regular
variation} \cite{BOst3}. Its generalizations to normed
groups\index{group!normed group} may be used to prove the Uniform Boundedness
Theorem\index{uniform boundedness theorem} (see \cite{Ost}). Recently it has
found applications to additive combinatorics in the area of Ramsey
Theory\index{Ramsey, F. P.} (for which see \cite{GRS,HS})\index{Graham,
R. L.}\index{Rothschild, B. L.}\index{Spencer, J. H.}\index{Hindman,
N.|)}\index{Strauss, D.}, best visualized in the language of \Index{colour}:
one seeks monochromatic structures in finitely coloured situations. Two
examples are included in Section 6.

The KBD theorem is about
\index{embedding!shift-embedding}shift-embedding of subsequences of
a null sequence $\{z_{n}\}$ into a \textit{single} set $T$ with an
assumption of
\textit{regularity} (Baire/measurable). Our generalizations in Section 3 of a
theorem of Kingman's have been motivated by the wish to establish `multiple
embedding' versions of KBD: we seek conditions on a sequence $\{z_{n}\}$ and a
\textit{family} of sets $\{T_{k}\}_{k\in \omega }$ which together guarantee
that \textit{one} shift embeds (different) subsequences of $\{z_{n}\}$ into
\textit{all} members of the family.

Evidently, if $t+z_{n}$ lies in several sets infinitely often, then the sets
in question have a common limit point, a sense in which they
are \index{contiguity|)}contiguous at $t$. Thus \textit{contiguity conditions}
are one goal, the other two being
\textit{regularity conditions} on the family, and \textit{admissibility
conditions} on the null sequences.

We view the original Kingman\index{Kingman, J. F. C.} Theorem as studying
contiguity at infinity, so that divergent sequences $z_{n}$ (i.e.\ with
$z_{n}\rightarrow +\infty )$ there replace the null sequences of KBD\null. The
theorem uses
\textit{openness} as a regularity condition on the
family, \textit{cofinality}\index{cofinality} at infinity
(e.g.\ \textit{unboundedness} on the right) as the simplest contiguity
condition at infinity, and
\begin{equation}
\frac{z_{n+1}}{z_{n}}\rightarrow 1\text{ (multiplicative form)},\quad
z_{n+1}-z_{n}\rightarrow 0\text{ (additive form)}  \tag{$*$}  \label{*}
\end{equation}
as the admissibility condition on the divergent sequence $z_{n}$ (\eqref{*}
follows from \Index{regular variation} by Weissman's\index{Weissman, I.}
Lemma, \cite[Lemma 1.9.6]{BGT})\index{Goldie, C. M.|)}\index{Teugels, J. L.|)}.  Taken together, these three guarantee
multiple embedding (at infinity).

One can switch from $\pm \infty $ to 0 by an inversion $x\rightarrow 1/x$,
and thence to any $\tau $ by a shift $y\rightarrow y+\tau$. \textit{Openness}
remains the \textit{regularity} condition, a property of
\textit{density at zero }becomes the analogous \textit{admissibility}
condition on null sequences, and cofinality (or accumulation) at $\tau $ the
contiguity condition. The transformed theorem then asserts that for
admissible null sequences $\zeta _{n}$ there exists a scalar $\sigma $ such
that the sequence $\sigma \zeta _{n}+\tau $ has subsequences in all the open
sets $T_{k}$ provided these all accumulate at $\tau $.

In the next section, we will replace Kingman's\index{Kingman, J. F. C.}
regularity condition of openness by the Baire property\index{Baire,
R.-L.!Baire property}, or alternatively measurability, to obtain two versions
of Kingman's theorem---one for measure and one for category.  We develop the
regularity theme bitopologically, working with two topologies, so as to deduce
the measure case from the Baire case by switching from the Euclidean to the
density topology\index{topology!density topology|)}.

\begin{DN}We use
the notation $B_{r}(x):=\{y:|x-y|<r\}$ and $\omega :=\{0,1,2,\dots\}$.  Likewise
for $a\in A\subseteq \mathbb{R}$ and metric $\rho =\rho _{A}$ on $A$,
$B_{r}^{\rho }(a):=\{y\in A:\rho (a,y)<r\}$ and $\cl_{A}$ denotes closure in
$A$. For $S$ given, put $S^{>m}=S\backslash B_{m}(0)$. $\mathbb{R}_{+}$
denotes the (strictly) positive reals. When we regard $\mathbb{R}_{+}$ as a
multiplicative group, we write
\[
A\cdot B:=\{ab:a\in A,\ b\in B\},\qquad A^{-1}:=\{a^{-1}:a\in A\},
\]
for $A$, $B$ subsets of $\mathbb{R}_{+}$.

Call a Baire set $S$ \textit{essentially unbounded}\index{essentially (un)bounded} if for each $m\in
\mathbb{N}$ the set $S^{>m}$ is non-meagre\index{meagre}. This may be interpreted in the
sense of the metric (Euclidean) topology, or as we see later in the measure
sense by recourse to the density topology. To distinguish the two, we will
qualify the term by referring to the category/metric or the measure sense.

Say that a set $S\subset \mathbb{R}_{+}$ \textit{accumulates
essentially}\index{accumulate essentially} at $0$ if $S^{-1}$ is essentially
unbounded. (In \cite{BHW} such sets are called measurably/Baire \textit{large}
at $0$.) Say that $S\subset \mathbb{R}_{+}$ \textit{accumulates essentially}
at $t$ if $(S-t)\cap \mathbb{R}_{+}$ accumulates essentially at $0$.
\end{DN}

We turn now to some topological notions. Recall (see e.g. \cite[4.3.23 and
24]{Eng}) that a metric space $A$ is \textit{completely
metrizable}\index{metrizability!complete metrizability} iff it is
a $\mathcal{G}_{\delta }$ subset\index{Gdelta@$G_\delta$!set} of its \Index{completion} (i.e.\ $A=\bigcap_{n\in
\omega }G_{n}$ with each $G_{n}$ open in the completion of $A$), in which
case it has an \Index{equivalent metric} under which it is complete. Thus a
$\mathcal{G}_{\delta }$ subset $A$ of the line has a metric $\rho =\rho _{A}$,
equivalent to the Euclidean metric, under which it is complete. (So for each
$a\in A$ and $\varepsilon >0$ there exists $\delta >0$ such that $B_{\delta
}(a)\subseteq B_{\varepsilon }^{\rho }(a)$, which enables the construction
of sequences with $\rho $-limit guaranteed to be in $A$.)

This motivates the definition below, which allows us to capture a feature of
measure-category duality: both exhibit $\mathcal{G}_{\delta
}$ \textit{inner-regularity}, modulo sets which we are prepared to
neglect. (The definition here takes advantage of the fact that $\mathbb{R}$ is
complete; for the general metric group\index{group!metric group} context see \cite[Section 5]{BOst6}.)

\begin{defn}Call $A\subset \mathbb{R}$ \textit{almost complete}\index{almost complete} (in
category/measure) if
\begin{enumerate}[(i)]
\item there is a \Index{meagre} set $N$ such that $A\backslash N$ is a
$\mathcal{G}_{\delta }$, or

\item for each $\varepsilon >0$ there is a measurable set $N$ with
$|N|<\varepsilon $ and $A\backslash N$ a $\mathcal{G}_{\delta }$.
\end{enumerate}
\end{defn}

Thus $A$ almost complete is Baire resp.\ measurable. A bounded non-null
measurable subset $A$ is almost complete: for each $\varepsilon >0$ there is
a compact (so $\mathcal{G}_{\delta }$) subset $K$ with $|A\backslash
K|<\varepsilon $, so we may take $N=A\backslash K$. Likewise a Baire
non-meagre set is almost complete---this is in effect a restatement of
Baire's Theorem\index{Baire, R.-L.!Baires theorem@Baire's theorem}:

\begin{Baire}For $A\subseteq \mathbb{R}$ Baire non-meagre there is a
meagre set $M$ such that $A\backslash M$ is completely metrizable.
\end{Baire}

\begin{proof}For $A\subseteq \mathbb{R}$ Baire non-meagre we have
$A=(U\backslash M_{0})\cup M_{1}$ with $M_{i}$ meagre and $U$ a non-empty
open set. Now $M_{0}=\bigcup_{n\in \omega }N_{n}$ with $N_{n}$ \Index{nowhere
dense}; the closure $F_{n}:=\bar{N}_{n}$ is also nowhere dense (and the
complement $E_{n}=\mathbb{R}\backslash F_{n}$ is dense, open). The set
$M_{0}^{\prime }=\bigcup_{n\in \omega }F_{n}$ is also meagre, so
$A_{0}:=U\backslash M_{0}^{\prime }=\bigcap_{n\in \omega }U\cap
E_{n}\subseteq A$. Taking $G_{n}:=U\cap E_{n}$, we see that $A_{0}$ is
completely metrizable.
\end{proof}

The tool whereby we interpret measurable functions as Baire functions is
\textit{refinement} of the usual metric (Euclidean) topology of the line
$\mathbb{R}$ to a non-metric one: the \textit{density
topology}\index{topology!density topology|(} (see e.g.\
\cite{Kech,LMZ,CLO}). Recall that for $T$ measurable, $t$ is
a (metric) \index{density point|(}density point of $T$ if
$\lim_{\delta \rightarrow 0}|T\cap I_{\delta }(t)|/\delta =1$, where
$I_{\delta }(t)=(t-\delta /2,t+\delta /2)$.  By the Lebesgue Density Theorem\index{Lebesgue, H. L.!Lebesgue density theorem}
almost all points of $T$ are density points (\cite[Section
61]{Hal}, \cite[Th.\ 3.20]{Oxt}, or \cite{Goff}). A set $U $ is $d$-open
(density-open = open in the density topology $d$) if (it is measurable and)
each of its points is a density point of $U$. We mention five properties:

\begin{enumerate}[(i)]
\item The density topology is finer than (contains) the Euclidean topology
\cite[17.47(ii)]{Kech}.

\item A set is Baire in the density topology iff it is (Lebesgue) measurable
\cite[17.47(iv)]{Kech}.

\item A Baire set is meagre in the density topology iff it is null
\cite[17.47(iii)]{Kech}. So (since a countable union of null sets is null), the
conclusion of the Baire theorem holds for the line under $d$.

\item $(\mathbb{R},d)$ is a \textit{Baire} space\index{Baire, R.-L.!Baire
space}, i.e.\ the conclusion of the Baire theorem holds
(cf.\ \cite[3.9]{Eng}).

\item A function is $d$-continuous iff it is approximately
continuous\index{approximate continuity} in
Denjoy's\index{Denjoy, A.} sense (\cite{Den}; \cite[pp.\ 1, 149]{LMZ}).
\end{enumerate}

The reader unfamiliar with the density topology may find it helpful to recall
Littlewood's\index{Littlewood, J. E.} Three Principles (\cite[Ch.\
4]{Lit}, \cite[Section 3.6, p.\ 72]{Roy}): general situations are `nearly' the
easy situations---i.e.\ are easy situations modulo small
sets. Theorem \ref{t:0} below is in this spirit. We refer now to Littlewood's
Second Principle, of a measurable function being continuous on nearly all of
its domain, in a form suited to our $d$-topology context.

\begin{Lusin}\index{Luzin, N. N.}For $f:\mathbb{R}_{+}\rightarrow \mathbb{R}$ measurable,
there is a density-open set $S$ which is almost all of $\mathbb{R}_{+}$
and an increasing decomposition $S:=\bigcup_{m}S_{m}$
into density-open sets $S_{m}$ such that each $f|S_{m}$
is continuous in the usual sense.
\end{Lusin}

\begin{proof} By a theorem of Lusin (see e.g.\ \cite{Hal}), there is an
increasing sequence of (non-null) compact sets $K_{n}$ ($n=1$, 2, \dots)
covering almost all of $\mathbb{R}_{+}$ with the function $f$ restricted to
$K_{n}$ continuous on $K_{n}$. Let $S_{n}$ comprise the density points of
$K_{n}$, so $S_{n}$ is density-open, is almost all of $K_{n}$ (by the
Lebesgue Density Theorem\index{Lebesgue, H. L.!Lebesgue density theorem}) and $f$ is continuous on $S_{n}$. Put
$S:=\bigcup_{m}S_{m}$; then $S$ is almost all of $\mathbb{R}_{+}$, is
density-open and $f|S_{m}$ is continuous for each $m$.
\end{proof}

Two results, Theorem S below and Theorem \ref{t:1} in the next section, depend
on the following consequence of Steinhaus's\index{Steinhaus, H. D.} theorem
concerning the existence of interior points of $A\cdot B^{-1}$ (\cite{St},
cf.\ \cite[Th.\ 1.1.1]{BGT})\index{Goldie, C. M.}\index{Teugels, J. L.} for $A$, $B$ measurable non-null. The first is in
multiplicative form a sharper version of Sierpi\'{n}ski's result that any two
non-null measurable sets realize a rational distance.

\begin{MS}\index{Sierpinski, W.@Sierpi\'nski, W.|(}For $a$, $b$ density
points\index{density point|)} of their respective measurable sets $A$, $B$ in
$\mathbb{R}_{+}$ and for $n=1$, $2$, \dots, there exist positive rationals
$q_{n}$ and points $a_{n}$, $b_{n}$ converging to $a$, $b$ through $A$, $B$
respectively such that $b_{n}=q_{n}a_{n}$.
\end{MS}

\begin{proof}For $n=1$, 2, \dots\ and the consecutive values $\varepsilon =1/n$
the sets $B_{\varepsilon }(a)\cap A$ and $B_{\varepsilon }(b)\cap B$ are
measurable non-null, so by Steinhaus's\index{Steinhaus, H. D.} theorem the set
$[B\cap B_{\varepsilon }(b)]\cdot \lbrack A\cap B_{\varepsilon }(a)]^{-1}$
contains interior points and so in particular a rational point $q_{n}$. Thus
for some $a_{n}\in B_{\varepsilon }(a)\cap A$ and $b_{n}\in B_{\varepsilon
}(b)\cap B$ we have $q_{n}=b_{n}a_{n}^{-1}$, and as $|a-a_{n}|<1/n$ and
$|b-b_{n}|<1/n$, $a_{n}\rightarrow a$, $b_{n}\rightarrow b$.
\end{proof}

\begin{rems}
\begin{enumerate}[1.]
\item For the purposes of Theorem \ref{t:2} below, we observe that
$q_{n}$ may be selected arbitrarily large, for fixed $a$, by taking $b$
sufficiently large (since $q_{n}\rightarrow ba^{-1}$).

\item The Lemma addresses $d$-open sets but also holds in the metric topology
(the proof is similar but simpler), and so may be restated
bitopologically (from the viewpoint of \cite{BOst7}) as follows.
\end{enumerate}
\end{rems}

\begin{Sier}\index{Sierpinski, W.@Sierpi\'nski, W.|)}For
$\mathbb{R}_{+}$ with either the Euclidean or the density
topology\index{topology!density topology|)},
if $a$, $b$ are respectively in the open sets $A$, $B$,
then for $n=1$, $2$, \dots\ there exist positive rationals
$q_{n}$ and points $a_{n}$, $b_{n}$ converging metrically to
$a$, $b$ through $A$, $B$ respectively such that
\[
b_{n}=q_{n}a_{n}.
\]
\end{Sier}

\section{A bitopological Kingman theorem}\index{Kingman, J. F. C.|(}
We begin by simplifying essential unboundedness modulo null/meagre sets.

\addtocounter{theorem}{-1}

\begin{theorem}\label{t:0}In $\mathbb{R}_{+}$ with the Euclidean
or density topology, for $S$ Baire/measurable and
\index{essentially (un)bounded}essentially unbounded there exists an open/\break density-open unbounded $G$
and meagre/null $M$ with $G\backslash M\subset S$.
\end{theorem}

\begin{proof}Choose integers $m_{n}$ inductively with $m_{0}=0$ and
$m_{n+1}>m_{n}$ the least integer such that $(m_{n},m_{n+1})\cap S$ is
non-meagre; for given $m_{n}$ the integer $m_{n+1}$ is well-defined, as
otherwise for each $m>m_{n}$ we would have $(m_{n},m)\cap S$ meagre, and so
also
\[
(m_{n},\infty )\cap S=\bigcup_{m>m_{n}}(m_{n},m)\cap S\quad\text{meagre},
\]
contradicting $S$ essentially unbounded. Now, as $(m_{n},m_{n+1})\cap S$ is
Baire/ measurable, we may choose $G_{n}$ open/density-open and
$M_{n}$, $M_{n}^{\prime }$ meagre subsets of $(m_{n},m_{n+1})$ such that
\[
((m_{n},m_{n+1})\cap S)\cup M_{n}=G_{n}\cup M_{n}^{\prime }.
\]
Hence $G_{n}$ is non-empty. Put $G:=\bigcup_{n}G_{n}$ and
$M:=\bigcup_{n}M_{n}$. Then $M$ is meagre and $G$ is open unbounded and,
since $M\cap (m_{n},m_{n+1})=M_{n}$ and $G\cap (m_{n},m_{n+1})=G_{n}$,
\[
G\backslash M=\bigcup_{n}G_{n}\backslash
M=\bigcup_{n}G_{n}\backslash M_{n}\subset
\bigcup_{n}(m_{n},m_{n+1})\cap S=S,
\]
as asserted.
\end{proof}

\begin{WA}\index{Archimedes of Syracuse!weakly Archimedean property|(}Let
$\mathbb{I}$ be $\mathbb{N}$ or $\mathbb{Q}$, with the ordering induced from
the reals. Our purpose will be to take limits through subsets $J$ of
$\mathbb{I}$ which are \textit{unbounded} on the right (more briefly:
unbounded). According as $\mathbb{I}$ is $\mathbb{N}$ or $\mathbb{Q}$, we will
write $n\rightarrow \infty $ or $q\rightarrow \infty $. Denote by $X$ the line
with either the metric or density topology and say that a family
$\{h_{i}:i\in \mathbb{I}\}$ of self-homeomorphisms\index{homeomorphism|(} of the
topological space $X$ is \textit{weakly Archimedean} if for each non-empty
open set $V$ in $X$ and any $j\in \mathbb{I}$ the open set
\[
U_{j}(V):=\bigcup_{i\geq j}h_{i}(V)
\]
meets every essentially unbounded set in $X$.
\end{WA}

\begin{theorem}[implicit in {\cite[Th.\ 1.9.1(i)]{BGT}}]\label{t:1}\index{Goldie, C. M.|(}\index{Teugels, J. L.|(}
In the multiplicative group of positive reals $\mathbb{R}_{+}^{\ast }$
with Euclidean topology, the functions $h_{n}(x)=d_{n}x$
for $n=1$, $2$, \dots, are homeomorphisms and $\{h_{n}:n\in N\}$
is weakly Archimedean, if $d_{n}$ is divergent and the
multiplicative form of\/ {\rm (\ref{*})} holds. For any interval
$J=(a,b)$ with $0<a<b$ and any $m$,
\[
U_{m}(J):=\bigcup_{n\geq m}d_{n}J
\]
contains an infinite half-line, and so meets every unbounded open
set. Similarly this is the case in the additive group of reals $\mathbb{R}$
with $h_{n}(x)=d_{n}+x$ and $U_{m}(J)=\bigcup_{n\geq m}(d_{n}+J)$.
\end{theorem}

\begin{proof}For given $\varepsilon >0$ and all large enough $n$,
$1-\varepsilon <d_{n}/d_{n+1}<1+\varepsilon $. Write $x:=(a+b)/2\in
J$. For
$\varepsilon $ small enough $a<x(1-\varepsilon )<x(1+\varepsilon )<b$, and
then $a<xd_{n}/d_{n+1}<b$, so $xd_{n}\in d_{n+1}J$, and so $d_{n}J$ meets
$d_{n+1}J$. Thus for large enough $n$ consecutive $d_{n}J$ overlap; as
$d_{n}\rightarrow \infty $, their union is thus a half-line.
\end{proof}

\begin{rem}Some such condition as (\ref{*}) is necessary, otherwise the set
$U_{m}(J)$ risks missing an unbounded sequence of open intervals. For an
indirect example, see the remark in \cite{BGT} after Th.\ 1.9.2 and G. E. H.
Reuter's\index{Reuter, G. E. H.} elegant counterexample to a corollary of
Kingman's Theorem, a break-down caused by the absence of our condition. For a
direct example, note that if $d_{n}=r^{n}\log n$ with $r>1$ and $J=(0,1)$,
then $d_{n}+J$ and $d_{n+1}+J$ miss the interval $(1+r^{n}\log n,r^{n+1}\log
(1+n))$ and the omitted intervals have union an unbounded open set; to see
that the omitted intervals are non-degenerate note that their lengths are
unbounded:
\[
r^{n+1}\log (1+n)-r^{n}\log n-1\rightarrow \infty .
\]
\end{rem}

Theorem \ref{t:1} does not extend to the real line under the density topology;
the homeomorphisms\index{homeomorphism|)} $h_{n}(x)=nx$ are no longer weakly Archimed\-ean, as we
demonstrate by an example in Theorem \ref{t:12}. We are thus led to an
alternative approach:

\begin{theorem}\label{t:2}In the multiplicative group of reals
$\mathbb{R}_{+}^{\ast }$ with
the density topology, the family of homeomorphisms $\{h_{q}:q\in
\mathbb{Q}_{+}\}$ defined by $h_{q}(x):=qx$, where $\mathbb{Q}_{+}$ has its
natural order, is weakly Archimedean\index{Archimedes of Syracuse!weakly
Archimedean property|)}. For any density-open set $A$ and any
$j\in \mathbb{Q}_{+}$,
\[
U_{j}(A):=\bigcup_{q\geq j,\,q\in \mathbb{Q}_{+}}qA
\]
contains almost all of an infinite half-line, and so meets every
unbounded density-open set.
\end{theorem}

\begin{proof}Let $B$ be Baire and essentially
unbounded in the $d$-topology.  Then $B$ is measurable and essentially
unbounded in the sense of measure. From Theorem \ref{t:0}, we may assume that
$B$ is density-open. Let $A$ be non-empty density-open. Fix $a\in A$ and $j\in
$ $\mathbb{Q}_{+}$. Since $B$ is unbounded, we may choose $b\in B$ such that
$b>ja$. By Theorem S there is a $q\in \mathbb{Q}_{+}$ with $j<q<ba^{-1}$ such
that $qa^{\prime }=b^{\prime }$, with $a^{\prime }\in A$ and $b^{\prime }\in
B$. Thus
\[
U_{j}(A)\cap B\supseteq h_{q}(A)\cap B=qA\cap B\neq \varnothing ,
\]
as required.

If $U_{j}(A)$ fails to contain almost all of any infinite half-line, then
its complement $B:=\mathbb{R}_{+}\backslash U_{j}(A)$ is essentially
unbounded in the sense of measure and so, as above, must meet $U_{j}(A)$, a
contradiction.
\end{proof}

\begin{rems}
\begin{enumerate}[1.]
\item For $A$ the set of irrationals in $(0,1)$ the set
$U_{j}(A)$ is again a set of irrationals which contains almost all, but not
all, of an infinite half-line. Our result is thus best possible.

\item Note that (\ref{*}) is relevant to the distinction between integer and
rational skeletons\index{skeleton!rational skeleton}\index{skeleton!integer skeleton}; see the
prime-divisor\index{prime divisor} example on p.~53
of \cite{BGT}. Theorem \ref{t:2} holds with $\mathbb{Q}_{+}$ replaced by any
countable \textit{dense} subset of $\mathbb{R}_{+}^{\ast }$, although later we
use the fact that $\mathbb{Q}_{+}$ is closed under multiplication. There is an
affinity here with the use of a dense `skeleton set\index{skeleton!set}' in the Heiberg--Seneta\index{Heiberg, C.}\index{Seneta, E.} Theorem,
Th.\ 1.4.3 of \cite{BGT}\index{Goldie, C. M.|)}\index{Teugels, J. L.|)}, and its extension Th.\ 3.2.5 therein.
\end{enumerate}
\end{rems}

Kingman's Theorem below, like KBD, is a generic assertion about embedding
into target sets. We address first the source of this genericity: a property
inheritable by supersets either holds generically or fails outright. This is
now made precise.

\begin{defn}Recall that $X$ denotes $\mathbb{R}_{+}$ with Euclidean
or density topology\index{topology!density topology}. Denote by
$\mathcal{B}a(X)$, or just $\mathcal{B}a$, the Baire sets
of the space $X$, and recall these form a $\sigma $-algebra.
Say that a correspondence $F:\mathcal{B}a\rightarrow \mathcal{B}a$ is
\textit{monotonic} if $F(S)\subseteq F(T)$ for $S\subseteq T$.
\end{defn}

The nub is the following simple result, which we call the Generic Dichotomy
Principle.

\begin{theorem}[\textbf{Generic Dichotomy Principle}]\label{t:3}For $F:\mathcal{B}
a\rightarrow \mathcal{B}a$ monotonic: either
\begin{enumerate}[{\rm (i)}]
\item there is a non-meagre\index{meagre} $S\in \mathcal{B}a$ with $S\cap
F(S)=\varnothing $, or
\item for every non-meagre $T\in \mathcal{B}a$, $T\cap
F(T)$ is quasi all of $T$.
\end{enumerate}

Equivalently: the existence condition that $S\cap F(S)\neq
\varnothing $ should hold for all non-meagre $S\in \mathcal{B}a$
implies the genericity condition that, for each non-meagre $T\in
\mathcal{B}a$, $T\cap F(T)$ is quasi all of $T$.
\end{theorem}

\begin{proof}Suppose that (i) fails. Then $S\cap F(S)\neq \varnothing $
for every non-meagre $S\in \mathcal{B}a$. We show that (ii) holds.
Suppose otherwise; thus for some $T$ non-meagre in $\mathcal{B}a$, the set
$T\cap F(T)$ is not almost all of $T$. Then the set $U:=T\backslash
F(T)\subseteq T$ is non-meagre (it is in $\mathcal{B}a$ as $T$ and $F(T)$
are) and so
\begin{align*}
\varnothing &\neq U\cap F(U)\qquad (S\cap F(S)\neq \varnothing
    \text{ for every non-meagre }S)\\
  &\subseteq U\cap F(T)\qquad (U\subseteq T\text{ and }F\text{ monotonic}).
\end{align*}
But as $U:=T\backslash F(T)$, $U\cap F(T)=\varnothing $, a contradiction.

The final assertion simply rephrases the dichotomy as an implication.
\end{proof}

The following corollary transfers the onus of verifying the existence
condition of Theorem \ref{t:3} to topological completeness.

\begin{theorem}[\textbf{Generic Completeness Principle}]\label{t:4}For $F:
\mathcal{B}a\rightarrow \mathcal{B}a$ monotonic, if $W\cap
F(W)\neq \varnothing $ for all non-meagre $W\in \mathcal{G}_{\delta }$
then, for each non-meagre $T\in \mathcal{B}a$,
$T\cap F(T)$ is quasi all\index{quasi all} of $T$.

That is, either
\begin{enumerate}[{\rm (i)}]
\item there is a non-meagre $S\in \mathcal{G}_{\delta }$
with $S\cap F(S)=\varnothing $, or
\item for every non-meagre $T\in \mathcal{B}a$, $T\cap
F(T)$ is quasi all of $T$.
\end{enumerate}
\end{theorem}

\begin{proof}From Theorem B, for $S$ non-meagre in $\mathcal{B}a$ there
is a non-meagre $W\subseteq S$ with $W\in \mathcal{G}_{\delta }$. So $W\cap
F(W)\neq \varnothing $ and thus $\varnothing \neq W\cap F(W)\subseteq S\cap
F(S)$, by monotonicity. By Theorem \ref{t:3} for every non-meagre $T\in
\mathcal{B} a$, $T\cap F(T)$ is quasi all of $T$.
\end{proof}

\begin{rems}In regard to the role of
$\mathcal{G}_{\delta }$\index{Gdelta@$G_\delta$!set} sets, we note
\textit{Solecki's analytic dichotomy
theorem}\index{Solecki, S.} (reformulating and generalizing a specific
instance discovered by Petruska\index{Petruska, Gy.}, \cite{Pet}) as
follows. For $\mathcal{I}$ a family of closed sets (in any \Index{Polish
space}), let $\mathcal{I}_{\text{ext}}$ denote the sets covered by a countable
union of sets in $\mathcal{I}$. Then, for $A$ an analytic set, either $A\in
\mathcal{I}_{\text{ext}}$, or $A$ contains a $\mathcal{G}_{\delta }$ set not
in $\mathcal{I}_{\text{ext}}$. See \cite{Sol1}, where a number of classical
theorems, asserting that a `large' \Index{analytic set} contains a `large'
compact subset, are deduced, and also \cite{Sol2} for further applications of
dichotomy. A superficially similar, but more distant result, is
\textit{Kuratowski's
Dichotomy}---(\cite{Kur-B}, \cite{Kur-1}\index{Kuratowski, K.},
\cite[Cor.\ 1]{McSh}): suppose a set $H$ of autohomeomorphisms\index{homeomorphism} acts transitively
on a space $X$, and $Z\subseteq X$ is Baire and has the property that for
each $h\in H$
\[
Z=h(Z)\text{ or }Z\cap h(Z)=\varnothing ,
\]
i.e.\ under each $h\in H$, either $Z$ is invariant or $Z$ and its image are
disjoint. Then, either $H$ is meagre or it is \Index{clopen}.
\end{rems}

\begin{exas}Here are four examples of monotonic correspondences. The
first two relate to standard results. The following two are canonical for
the current paper as they relate to KBD and to Kingman's Theorem in its
original form. Each correspondence $F$ below gives rise to a correspondence
$\Phi (A):=F(A)\cap A$ which is a lower or upper density and arises in the
theory of \textit{lifting}\index{lifting} \cite{IT1,IT2} and category
measures\index{category measure}
\cite[Th.\ 22.4]{Oxt}, and so gives rise to a fine
topology\index{topology!fine topology} on the real line.
See also \cite[Section 6F]{LMZ} on lifting topologies.

\begin{enumerate}[1.]
\item Here we apply Theorem \ref{t:3} to the real line with the density
  topology, in which the meagre sets are the null sets. Let $\mathcal{B}$
  denote a countable basis of Euclidean open neighbourhoods. For any set $T$
  and $0<\alpha <1$ put
\[
\mathcal{B}_{\alpha }(T):=\{I\in \mathcal{B}:|I\cap T|>\alpha |I|\},
\]
which is countable, and
\[
F(T):=\bigcap_{\alpha \in \mathbb{Q}\cap (0,1)}\bigcup \{I:I\in
\mathcal{B}_{\alpha }(T)\}.
\]
Thus $F$ is increasing in $T$, $F(T)$ is measurable (even if $T$ is not) and
$x\in F(T)$ iff $x$ is a density point of $T$. If $T$ is measurable, the set
of points $x$ in $T$ for which $x\in I\in \mathcal{B}$ implies that $|I\cap
T|<\alpha |I|$ is null (see \cite[Th.\ 3.20]{Oxt}). Hence any non-null
measurable set contains a density point. It follows that almost all points of
a measurable set $T\ $are density points. This is the Lebesgue Density
Theorem\index{Lebesgue, H. L.!Lebesgue density theorem} (\cite[Th.\ 3.20]{Oxt},
or \cite[Section 3.5]{Kucz}).

\item In \cite[Th.\ 2]{PWW} a category analogue of the Lebesgue Density
Theorem is established. This follows more simply from our Theorem \ref{t:3}.

\item For KBD, let $z_{n}\rightarrow 0$ and put $F(T):=\bigcap_{n\in
\omega }\bigcup_{m>n}(T-z_{m})$. Thus $F(T)\in \mathcal{B}a$ for $T\in
\mathcal{B}a$ and $F$ is monotonic. Here $t\in F(T)$ iff there is an
infinite $\mathbb{M}_{t}$ such that $\{t+z_{m}:m\in
\mathbb{M}_{t}\}\subseteq T$. The Generic Dichotomy Principle asserts that
once we have proved (for which see Theorem \ref{t:15} below) that an arbitrary
non-meagre set $T$ contains a `translator', i.e.\ an element $t$ which
shift-embeds a subsequence $z_{m}$ into $T$, then quasi all elements of $T$
are translators.

\item For $z_{n}=n$ and $\{S_{k}\}$ a family of unbounded open sets (in the
  Euclidean sense), put $F(T):=T\cap \bigcap_{k\in \omega }\bigcap_{n\in
    \omega }\bigcup_{m>n}(S_{k}-z_{m})$. Thus $F(T)\in \mathcal{B}a$ for $T\in
  \mathcal{B}a$ and $F$ is monotonic. Here $t\in F(T)$ iff $t\in T$ and for
  each $k\in \omega $, there is an infinite $\mathbb{M}_{t}^{k}$ such that
  $\{t+z_{m}:m\in \mathbb{M}_{t}^{k}\}\subseteq S_{k}$. In Kingman's version
  of his theorem, as stated, we know only that $F(V)$ is non-empty for any
  non-empty open set $V$; but in Theorem \ref{t:5} below we adjust his
  argument to show that $F(T)$ is non-empty for arbitrary non-meagre sets
  $T\in \mathcal{B}a$, hence that quasi all members of $T$ are in $F(T)$, and
  in particular that this is so for $T=\mathbb{R}_{+}$.
\end{enumerate}
\end{exas}

\begin{theorem}[\textbf{Bitopological Kingman
Theorem}---{\cite[Th.\ 1]{King1}}, \cite{King2}, where
$\mathbb{I=N}$]\label{t:5}If $X$ is a Baire space,
\begin{enumerate}[{\rm (i)}]
\item $\{h_{i}:i\in \mathbb{I}\}$ is a countable, linearly
ordered, weakly Archimedean family of self-homeomorphisms\index{homeomorphism} of $X$, and
\item $\{S_{k}:k=1,2,\dots\}$ are essentially unbounded Baire sets,
\end{enumerate}
then for quasi all $\eta \in X$ and all $k\in
\mathbb{N}$ there exists an unbounded subset $\mathbb{J}_{\eta}^{k}$
of $\mathbb{I}$ with
\[
\{h_{j}(\eta ):j\in \mathbb{J}_{\eta }^{k}\}\subset S_{k}.
\]
Equivalently, if\/ {\rm (i)} and
\begin{enumerate}[${\rm (i)}'$]
\addtocounter{enumi}{1}
\item $\{A_{k}:k=1,2,\dots\}$ are Baire and all
accumulate essentially at $0$,
\end{enumerate}
then for quasi all $\eta $
and every $k=1$, 2, \dots\ there exists $\mathbb{J}_{\eta }^{k}$
unbounded with
\[
\{h_{j}(\eta )^{-1}:j\in \mathbb{J}_{\eta }^{k}\}\subset A_{k}.
\]
\end{theorem}

\begin{proof}We will apply Theorem \ref{t:3} (Generic Dichotomy), so consider an
arbitrary non-meagre Baire set $T$. We may assume without loss of generality
that $T=V\backslash M$ with $V$ non-empty open and $M$ meagre. For each $k=1$,
2, \dots\ choose $G_{k}$ open and $N_{k}$ and $N_{k}^{\prime }$ meagre such
that $S_{k}\cup N_{k}^{\prime }=G_{k}\cup N_{k}$. Put $N:=M\cup
\bigcup_{n,k}h_{n}^{-1}(N_{k}^{\prime })$; then $N$ is meagre (as $h_{n}$, and
so $h_{n}^{-1}$, is a homeomorphism\index{homeomorphism}).

As $S_{k}$ is essentially unbounded, $G_{k}$ is unbounded (otherwise, for
some $m$, $G_{k}\subset (-m,m)$, and so $S_{k}\cap (m,\infty )\subset N_{k}$
is meagre). Define the open sets $G_{jk}:=\bigcup_{i\geq
j}h_{i}^{-1}(G_{k})$. We first show that each $G_{jk}$ is dense. Suppose,
for some $j$, $k$, there is a non-empty open set $V$ such that $V\cap
G_{jk}=\varnothing $. Then for all $i\geq j$,
\[
V\cap h_{i}^{-1}(G_{k})=\varnothing ;\qquad G_{k}\cap h_{i}(V)=\varnothing .
\]
So $G_{k}\cap \bigcup_{i\geq j}h_{i}(V)=\varnothing $, i.e., for $U^{j}$
the open set $U^{j}:=\bigcup_{i\geq j}h_{i}(V)$ we have $G_{k}\cap
U^{j}=\varnothing $. But as $G_{k}$ is unbounded, this contradicts
$\{h_{i}\} $ being a weakly Archimedean family.

Thus the open set $G_{jk}$ is dense (meets every non-empty open set); so, as
$\mathbb{I}$ is countable, the $\mathcal{G}_{\delta }$ set
\[
H:=\bigcap_{k=1}^{\infty }\bigcap_{j\in \mathbb{I}}G_{jk}
\]
is dense (as $X$ is a Baire space). So as $V$ is a non-empty open subset we
may choose $\eta \in (H\cap V)\backslash N$. (Otherwise $N\cup (X\backslash
H)$ and hence $V$ is of \Index{first category}.) Thus $\eta \in T$ and for all
$k=1$, 2, \dots
\begin{equation}
\eta \in V\cap \bigcap_{j\in \mathbb{I}}\bigcup_{i\geq
j}h_{i}^{-1}(G_{k})\text{ and }\eta \notin N.  \tag{eta}  \label{eta}
\end{equation}
For all $m$, as $h_{m}(\eta )\notin h_{m}(N)$ we have for all $m$, $k$ that
$h_{m}(\eta )\notin N_{k}^{\prime }$. Using (\ref{eta}), for each $k$ select an
unbounded $\mathbb{J}_{\eta }^{k}$ such that for $j\in \mathbb{J}_{\eta
}^{k}$, $\eta \in h_{j}^{-1}(G_{k})$; for such $j$ we have $\eta \in
h_{j}^{-1}(S_{k})$. That is, for some $\eta \in T$ we have
\[
\{h_{j}(\eta ):j\in \mathbb{J}_{\eta }^{k}\}\subset S_{k}.
\]
Now
\[
F(T):=T\cap \bigcap_{k=1}^{\infty }\bigcap_{j\in \mathbb{I}}
\bigcup_{i\geq j}h_{i}^{-1}(G_{k})
\]
takes Baire sets to Baire sets and is monotonic. Moreover, $\eta \in F(T)$
iff $\eta \in T$ and for each $k$ there is an unbounded $\mathbb{J}_{\eta
}^{k}$ with $\{h_{j}(\eta ):j\in \mathbb{J}_{\eta }^{k}\}\subset S_{k}$. We
have just shown that $T\cap F(T)\neq \varnothing $ for $T$ arbitrary
non-meagre, so the Generic Dichotomy Principle implies that $X\cap F(X)$ is
quasi all of $X$, i.e.\ for quasi all $\eta $ in $X$ and each $k$
there is an unbounded $\mathbb{J}_{\eta }^{k}$ with $\{h_{j}(\eta ):j\in
\mathbb{J}_{\eta }^{k}\}\subset S_{k}$.
\end{proof}

Working in either the density or the Euclidean topology, we obtain
the following conclusions.

\refstepcounter{theorem}\label{t:6}
\begin{KCat}\index{Kingman, J. F. C.}If
$\{S_{k}:k=1,2,\dots\}$ are Baire and essentially unbounded in the
category sense, then for quasi all $\eta $ and each $k\in
\mathbb{N}$ there exists an unbounded subset $\mathbb{J}_{\eta }^{k}$
of $\mathbb{N}$ with
\[
\{n\eta :n\in \mathbb{J}_{\eta }^{k}\}\subset S_{k}.
\]
In particular this is so if the sets $S_{k}$ are open.
\end{KCat}

\begin{KMeas}\index{Kingman, J. F. C.}If
$\{S_{k}:k=1,2,\dots\}$ are measurable and essentially unbounded in
the measure sense, then for almost all $\eta $ and each $k\in
\mathbb{N}$ there exists an unbounded subset $\mathbb{J}_{\eta
}^{k}$ of $\mathbb{Q}_{+}$ with
\[
\{q\eta :q\in \mathbb{J}_{\eta }^{k}\}\subset S_{k}.
\]
\end{KMeas}

In the corollary below $\mathbb{J}_{t}^{k}$ refers to unbounded subsets of
$\mathbb{N}$ or $\mathbb{Q}_{+}$ according to the category/measure context.
It specializes down to a KBD result for a single set $T$ when $T_{k}\equiv T$,
but it falls short of KBD in view of the extra admissibility assumption and
the factor $\sigma $ (the latter an artefact of the multiplicative setting).

\begin{Cor}For $\{T_{k}:k\in \omega \}$
Baire/measurable and $z_{n}\rightarrow 0$ admissible, for
generically all $t\in \mathbb{R}$ there exist $\sigma _{t}$
and unbounded $\mathbb{J}_{t}^{k}$ such that for $k=1$, $2$, \dots
\[
t\in T_{k}\Longrightarrow \{t+\sigma _{t}z_{m}:m\in
\mathbb{J}_{t}^{k}\}\subset T_{k}.
\]
\end{Cor}

\begin{proof}For $T$ Baire/measurable, let $N=N(T)$ be the set of points
$t\in T$ that are not points of essential accumulation of $T$; then $t\in N$
if for some $n=n(t)$ the set $T\cap B_{1/n}(t)$ is meagre/null. As $\mathbb{R}$
with the Euclidean topology is (hereditarily) \index{second countable}second-countable, it is
hereditarily Lindel\"of\index{Lindelof, E. L.@Lindel\"of, E. L.} (see \cite[Th.\ 3.8.1]{Eng} or \cite[Th.\ 8.6.3]{Dug}),
so for some countable $S\subset N$
\[
N\subset \bigcup_{t\in S}T\cap B_{1/n(t)}(t),
\]
and so $N$ is meagre/null. Thus the set $N_{k}$ of points $t\in T_{k}$ such
that $T_{k}-t$ does not accumulate essentially at $0$ is meagre/null, as is
$N=\bigcup_{k}N_{k}$. For $t\notin N$, put $\Omega _{t}:=\{k\in \omega
:T_{k}-t\text{ accumulates essentially at }0\}$. Applying Kingman's Theorem to
the sets $\{T_{k}-t:k\in \Omega _{t}\}$ and the sequence $z_{n}\rightarrow
0$, there exist $\sigma_{t}$ and unbounded $\mathbb{J}_{t}^{k}$ such that
for $k\in \Omega_{t}$
\[
\{\sigma _{t}z_{m}:m\in \mathbb{J}_{t}^{k}\}\subset T_{k}-t,\text{ i.e.\ }
\{t+\sigma _{t}z_{m}:m\in \mathbb{J}_{t}^{k}\}\subset T_{k}.
\]
Thus for $t\notin N$, so for generically all $t$, there exist $\sigma _{t}$
and unbounded $\mathbb{J}_{t}^{k}$ such that for $k=1$, 2, \dots
\[
t\in T_{k}\Longrightarrow \{t+\sigma _{t}z_{m}:m\in
\mathbb{J}_{t}^{k}\}\subset T_{k}.\qedhere
\]
\end{proof}

\section{Applications---rational skeletons}\index{skeleton!rational skeleton|(}
In \cite{King1} Kingman's applications were concerned mostly with limiting
behaviour of continuous functions, studied by means of $h$-skeletons defined
by
\[
L_{\mathbb{N}}(h):=\lim_{n\rightarrow \infty }f(nh),
\]
assumed to exist for all $h$ in some interval $I$. This works for Baire
functions; but in our further generalization to measurable functions $f:
\mathbb{R}_{+}\mathbb{\rightarrow }\mathbb{R}$, we are led to study limits
$L_{\mathbb{Q}}(h):=\lim_{q\rightarrow \infty }f(qh)$, taken through the
rationals. Using the decomposition
\[
q:=n(q)+r(q),\qquad n(q)\in \omega ,\ r(q)\in \lbrack 0,1),
\]
the limit $L_{\mathbb{Q}}(h)$ may be reduced to, and so also computed as,
$L_{\mathbb{N}}(h)$ (provided we admit perturbations on $h$, making the
assumption of convergence here more demanding)---see Theorem \ref{t:11} below.

\begin{theorem}[\textbf{Conversion of sequential to continuous limits at
infinity}---cf.\ {\cite[Cor.\ 2 to Th.\ 1]{King1}}]\label{t:7}For
$f:\mathbb{R}_{+}\rightarrow \mathbb{R}$ measurable and $V$ a
non-empty, density-open set (in particular, an open interval), if
\[
\lim_{q\rightarrow \infty }f(qx)=0,\text{ for each }x\in V,
\]
then
\[
\lim_{t\rightarrow \infty }f(t)=0.
\]
The category version holds, mutatis mutandis, with $f$
Baire and $V$ open.
\end{theorem}

\begin{proof}Suppose not; choose $c>0$ with $\limsup_{t\rightarrow
\infty }|f(t)|>c>0$. By Theorem L there is a density-open set $S$ which is
almost all of $\mathbb{R}_{+}$ and an increasing decomposition
$S:=\bigcup_{m}S_{m}$ such that each $f|S_{m}$ is continuous. Put $B=\{s\in
S:|f(s)|>c\}$. For any $M>0$, there is an $s^{\ast }\in S_{m}$ for some $m$
with $s^{\ast }>M$ such that $|f(s^{\ast })|>c$. Then by continuity of
$f|S_{m}$, for some $\delta >0$ we have $|f(s)|>c$ for $s\in B_{\delta
}(s^{\ast })\cap S_{m}$. Thus $B$ is essentially unbounded. By Theorem
\ref{t:6}M there exists $v\in V$ such that $qv\in B$, for unboundedly many $q\in
\mathbb{Q}_{+}$; but, for such a $v$, we have $\lim_{q\rightarrow \infty
}f(qv)\neq 0$, a contradiction.
\end{proof}

We will need the following result, which is of independent interest. The Baire
case is implicit in \cite[Th.\ 2]{King1}. A related Baire category result is
in \cite[Th.\ 2.3.7]{HJ}\index{Hoffmann-Jorgensen, J.@Hoffmann-J\o rgensen,
J.} (with $G=\mathbb{R}_{+}$ and $T=\mathbb{Q}$ there).

\begin{theorem}[\textbf{Constancy of rationally invariant
functions}]\label{t:8}\index{rationally invariant function|(}If for
$f:\mathbb{R}_{+}\rightarrow \mathbb{R}$ measurable
\[
f(qx)=f(x),\text{ for }q\in \mathbb{Q}_{+}\text{ and almost all }x\in
\mathbb{R}_{+},
\]
then $f(x)$ takes a constant value almost everywhere. The
category version holds, mutatis mutandis, with $f$ Baire.
\end{theorem}

\begin{proof}[Proof (for the measure case)]Again by Theorem L, there is a
density-open set $S$ which is almost all of $\mathbb{R}_{+}$ and an
increasing decomposition $S:=\bigcup_{m}S_{m}$ such that each
$f|S_{m}$ is continuous. We may assume without loss of generality that
$f(qx)=f(x)$, for $q\in \mathbb{Q}_{+}$ and all $x\in S$. We claim that on $S$
the function $f$ is constant. Indeed, by Theorem S if $a$, $b$ are in $S_{m}$,
then, since they are density points of $S_{m}$, there are $a_{n}$, $b_{n}$ in
$S_{m}$ and $q_{n}$ in $\mathbb{Q}_{+}$ such that $b_{n}=q_{n}a_{n}$ with
$a_{n}\rightarrow a$ and $b_{n}\rightarrow b$, as $n\rightarrow \infty $ (in
the metric sense). Hence, since $f(q_{n}a_{n})=f(a_{n})$, relative
continuity gives
\[
f(b)=\lim_{n\rightarrow \infty}f(b_{n})
=\lim_{n\rightarrow \infty}f(q_{n}a_{n})=\lim_{n\rightarrow \infty }f(a_{n})=f(a).
\]
The Baire case is similar, but simpler.
\end{proof}

Of course, if $f(x)$ is the indicator function $1_{\mathbb{Q}}(x)$, which is
measurable/Baire, then $f(x)$ is constant almost everywhere, but not
constant, so the result in either setting is best possible.

\begin{theorem}[\textbf{Uniqueness of
limits}---cf.\ {\cite[Th.\ 2]{King1}}]\label{t:9}For
$f:\mathbb{R}_{+}\rightarrow \mathbb{R}$ measurable,
suppose that for each $x>0$ the limit
\[
L(x):=\lim_{q\rightarrow \infty }f(qx)
\]
exists and is finite on a density-open set $V$ (in particular an
interval). Then $L(x)$ takes a constant value a.e.\ in
$\mathbb{R}_{+}$, $L$ say, and
\[
\lim_{t\rightarrow \infty }f(t)=L.
\]
The category version holds, mutatis mutandis, with $f$
Baire and $V$ open.
\end{theorem}

\begin{proof}Since $\mathbb{Q}_{+}$ is countable, the function $L(x)$ is
measurable/Baire. Note that for $q\in \mathbb{Q}_{+}$ one has $L(qx)=L(x)$, so
if $L$ is defined for $x\in V$, then $L$ is defined for $x\in qV$ for each
$q\in \mathbb{Q}_{+}$, so by Theorem \ref{t:2} for almost all $x$ in some
half-infinite interval and thus for almost all $x\in \mathbb{R}_{+}$. The
result now follows from Theorem \ref{t:8}. As for the final conclusion,
replacing $f(x)$ by $f(x)-L$, we may suppose that $L=0$, and so may apply
Theorem \ref{t:7}.
\end{proof}

We now extend an argument in \cite{King1}. Recall that $f$ is
\textit{essentially bounded}\index{essentially (un)bounded} on $S$ if $\esssup_{S}f<\infty $.

\begin{defn}Call $f$ \textit{essentially bounded at infinity} if
for some $M$\break $\esssup_{(n,\infty )}f\leq M$, for all large enough $n$,
i.e.\ $\limsup_{n\rightarrow \infty }[\esssup_{(n,\infty )}\allowbreak
f]<\infty$.
\end{defn}

\begin{theorem}[\textbf{Essential boundedness
theorem}---cf.\ {\cite[Cor.\ 3 to Th.\ 1]{King1}}]\label{t:10}For $V$ non-empty
density-open and $f:\mathbb{R}_{+}\rightarrow \mathbb{R}$ measurable,
suppose that for each $x\in V$
\[
\sup \{f(qx):q\in \mathbb{Q}_{+}\}<\infty .
\]
Then $f(t)$ is essentially bounded at infinity. The
category version holds, mutatis mutandis, with $f$ Baire and $V$ open.
\end{theorem}

\begin{proof}Suppose not. Then for each $n=1$, 2, \dots\ there exists $m\in
\mathbb{N}$, arbitrarily large, such that $\esssup_{(m,\infty )}f>n$. We
proceed inductively. Suppose that $m(n)$ has been defined so that
$\esssup_{(m(n),\infty )}f>n$. Choose $m(n+1)>m(n)$ so that
\[
G_{n}:=\{t:m(n)<t<m(n+1)\text{ and }|f(t)|>n\}
\]
is non-null (otherwise, off a null set, we would have $|f(t)|\leq n$ for all
$t>m(n)$, making $n$ an essential bound of $f$ on $(m,\infty )$, for each
$m>m(n)$, contradicting the assumed essential unboundedness). Each $G_{n}$ is
measurable non-null, so defining
\[
G:=\bigcup_{n}G_{n}
\]
yields $G$ essentially unbounded. So there exists $v\in V$ such that $qv\in G$,
for an unbounded set of $q\in \mathbb{Q}_{+}$. Since each set $G_{n}$ is
bounded, the set $\{qv:q\in \mathbb{Q}_{+}\}$ meets infinitely many of the
disjoint sets $G_{n}$, and so
\[
\sup \{|f(qv)|:q\in \mathbb{Q}_{+}\}=\infty ,
\]
contradicting our assumption.
\end{proof}\index{Kingman, J. F. C.|)}

We close with the promised comparison of $L_{\mathbb{Q}}(h)$ with
$L_{\mathbb{N}}(h)$ (of course, if $L_{\mathbb{Q}}(h)$ exists, then so does
$L_{\mathbb{N}}(h)$ and they are equal). We use the decomposition $q=n(q)+r(q)$,
with $n(q)\in \mathbb{N}$ and $r(q)\in \lbrack 0,1)\cap \mathbb{Q}$.

\begin{theorem}[\textbf{Perturbed
skeletons}]\label{t:11}\index{skeleton!perturbed skeleton}For $f:\mathbb{R}_{+}
\rightarrow\mathbb{R}$ measurable, the limit
$L_{\mathbb{Q}}(h)$ exists for all $h$ in the non-empty interval
$I=(0,b)$ with $b>0$ iff for all $h$ in $I$ the limit
\[
\lim_{n}f(n(h+z_{n}))
\]
exists, for every null sequence $z_{n}$ with
\[
r_{n}:=n(z_{n}/h)\in \lbrack 0,1)\cap \mathbb{Q}.
\]
Furthermore, if either limit exists on $I$ then it exists
a.e.\ on $\mathbb{R}_{+}$, and then both limits are equal a.e.\ to
$L_{\mathbb{N}}(h)$.

If $I=(a,b)$ with $0<a<b$, the assertion holds
far enough to the right.
\end{theorem}

\begin{proof}First we prove the asserted equivalence.

Suppose the limit $L_{\mathbb{Q}}(h)$ exists for all $h$ in the interval $I$.
Then given $z_{n}$ as above, take $q_{n}:=n+r_{n}$, $r_{n}=n(z_{n}/h)$; then
$n(q_{n})=n$, $r(q_{n})=r_{n}$, and
\[
q_{n}h=n(h+z_{n}).
\]
So the following limit exists:
\[
\lim_{n}f(n(h+z_{n}))=\lim_{n}f(q_{n}h)=L_{\mathbb{Q}}(h).
\]

For the converse, take $z_{n}$ as above (so $z_{n}=r_{n}h/n$ with $r_{n}\in
\lbrack 0,1)\cap \mathbb{Q}$); our assumption is that $L(h,\{r_{n}\})$ exists
for all $h\in I$, where
\[
L(h,\{r_{n}\}):=\lim_{n}f(n(h+z_{n})).
\]
Write $\partial L(\{r_{n}\})$ for the `domain of $L$'---the set of $h$ for
which this limit exists. Thus $I\subseteq \partial L(\{r_{n}\})$. Let
$q_{n}\rightarrow \infty $ be arbitrary in $\mathbb{Q}_{+}$. So
$q_{n}=n(q_{n})+r(q_{n})$ and, writing $r_{n}:=r(q_{n})$ and
$z_{n}:=r_{n}h/n(q_{n})$, we find
\[
q_{n}h=n(q_{n})[h+r_{n}h/n(q_{n})]=n(q_{n})[h+z_{n}],
\]
and
\[
r_{n}=n(z_{n}/h)\in \lbrack 0,1)\cap \mathbb{Q}.
\]
By assumption, $\lim_{n}f(q_{n}h)=L(h,\{r_{n}\})$ exists for $h\in I$.
Restricting from $\{n\}$ to $\{pn\}$, the limit $\lim_{n}f(np(h+z_{n}))$
exists for each $p\in \mathbb{N}$, and
\[
L(h,\{r_n\})=\lim_nf(np(h+z_n))=\lim_nf(n(h'+z'_n)),
\]
where $z_{n}^{\prime }=pz_{n}$ and $h^{\prime }=ph$. So
\[
L(h,\{r_{n}\})=L(ph,\{r_{n}\}),
\]
as
\[
n(z_{n}^{\prime }/h^{\prime })=n(z_{n}/h)=r_{n}\in \lbrack 0,1)\cap \mathbb{Q}.
\]
That is, $L(ph,\{r_{n}\})$ exists for $p$ a positive integer, whenever
$L(h,\{r_{n}\})$ exists, and equals $L(h,\{r_{n}\})$. As $h/p\in I=(0,b)$ for
$h\in I$ and $L(h,\{r_{n}\})=L(p(h/p),\{r_{n}\})=L(h/p,\{r_{n}\})$,
$L(rh,\{r_{n}\})$ exists when\-ever $r$ is a positive rational. So the domain
$\partial L$ of $L$ includes all intervals of the form $rI$ for positive
rational $r$, and so includes the whole of $\mathbb{R}_{+}$. Moreover,
$L(\cdot,\{r_{n}\})$ is rationally invariant\index{rationally invariant
function|)}. But, since $f$ is measurable, $L$ is a measurable function on
$I\times \mathbb{Q}_{+}^{\omega }$. Now $\mathbb{Q}_{+}$ can be identified
with $\mathbb{N}\times \mathbb{N}$, and hence $\mathbb{Q}_{+}^{\omega }$ can
be identified with $\mathbb{N}^{\omega }$.  This in turn may be identified
with the irrationals $\mathcal{I}$ (see e.g.\ \cite[p.\ 9]{JR}). So $L$ is
measurable on $\mathbb{R}_{+}\times
\mathcal{I}$ and so, by Theorem \ref{t:8}, $L(\cdot,\{r_{n}\})$ is almost
everywhere constant.

This proves the equivalence asserted. For the final assertion, the argument in
the last paragraph shows that, given the assumptions, $L_{\mathbb{Q}}(h)$
exists for a.e.\ positive $h$; from here the a.e.\ equality is immediate. This
completes the case $I=(0,b)$. For $I=(a,b)$ with $a>0$, $\bigcup_{p
\in \mathbb{N}}pI$ contains some half-line $[c,\infty )$ by Theorem \ref{t:1}.
\end{proof}

The following example, due to R. O. Davies\index{Davies, R. O.|(}, clarifies why use of the natural
numbers and hence also of discrete skeletons\index{skeleton!discrete skeleton}
$L_{\mathbb{N}}(h)$ is inadequate in the measure setting. (We thank Roy Davies
for this contribution.)

\begin{theorem}\label{t:12}The open set
\[
G:=\bigcup_{m=1}^{\infty }(m-2^{-(m+2)},m)
\]
is disjoint for each $n=1$, $2$, \dots\ from the dilation $nF$
of the non-null closed set $F$ defined by
\[
F:=\left[\frac{1}{2},1\right]\backslash \left( \bigcup_{m=1}^{\infty
}\bigcup_{n=m}^{2m-1}\left(\frac{m}{n}-\frac{1}{n2^{m+2}},\frac{m}{n}
\right)\right).
\]
\end{theorem}

\begin{proof}Suppose not. Put $z_{m}:=2^{-(m+2)}$ and
\[
E:=\bigcup_{m=1}^{\infty }\bigcup_{n=m}^{2m-1}\left(\frac{m}{n}-
\frac{1}{n}z_{m},\frac{m}{n}\right).
\]
Then for some $n$, there are $f\in F$ and $g\in G$ such that $nf=g$. So for
some $m=1$, 2, \dots
\[
m-z_{m}<nf=g<m,\text{ i.e.\ }\frac{m}{n}-\frac{z_{m}}{n}<f<\frac{m}{n}.
\]
But as $1/2\leq f\leq 1$, we have $n/2\leq m$ and
\[
\frac{m}{n}-\frac{1}{n}z_{m}<1,\text{ i.e.\ }m-z_{m}<n.
\]
Thus $m\leq n\leq 2m$, yielding the contradiction that $f\notin F$. Put
\[
a_{m}:=\sum_{n=m}^{2m-1}\frac{1}{n},\text{ so that }\frac{1}{2}\leq
a_{m}\leq 1.
\]
Then
\[
|E|=\sum_{m=1}^{\infty }a_{m}z_{m}\leq \sum_{m=1}^{\infty
}2^{-(m+2)}=\frac{1}{4},
\]
and $|E|\geq 1/8$. Hence the complementary set $F$ has measure at least
$1/4$.
\end{proof}\index{skeleton!rational skeleton|)}\index{Davies, R. O.|)}

\section{KBD in van der Waerden style}\index{Waerden, B. L. van
der}\index{Kestelman, H.|(}\index{Borwein, D.|(}\index{Ditor, S. Z.|(}
Fix $p$. Let $z_{n}$ be a null sequence. We prove a generalization of KBD
inspired by the van der Waerden theorem on arithmetic
progressions\index{arithmetic progression} (see
Section 6). For this we need the notation
\[
t+\bar{z}_{pm}=t+z_{pm+1},t+z_{pm+1},\dots,t+z_{pm+p}
\]
as an abbreviation for a block of consecutive terms of the null sequence all
shifted by $t$. Our unified proof, based on the $\mathcal{G}_{\delta }$-inner
regularity\index{Gdelta@$G_\delta$!inner regularity} common to measure and category
noted in Section 2, is inspired by a technique in \cite{BHW}.

\begin{theorem}[\textbf{Kestelman--Borwein--Ditor Theorem---consecutive form};
\cite{BOst5}]\label{t:13}Let $\{z_{n}\}\rightarrow 0$ be a
null sequence of reals. If $T$ is Baire/Lebesgue-measurable and
$p\in \mathbb{N}$, then for generically all $t\in T$ there
is an infinite set $\mathbb{M}_{t}$ such that
\[
\{t+\bar{z}_{pm}:m\in \mathbb{M}_{t}\}:=
\{t+z_{pm+1},t+z_{pm+1},\dots,t+z_{pm+p}:m\in \mathbb{M}_{t}\}\subseteq T.
\]
\end{theorem}

This will follow from the two results below, both important in their own
right. The first and its corollary address displacements of open sets in the
density and the Euclidean topologies; it is mentioned in passing in a note
added in proof (p.\ 32) in Kemperman \cite[Th.\ 2.1, p.\ 30]{Kem}, for which we
give an alternative proof. The second parallels an elegant result for the
measure case treated in \cite{BHW}.\index{Kestelman, H.|)}\index{Borwein,
D.|)}\index{Ditor, S. Z.|)}

\begin{Kemp}\index{Kemperman, J. H. B.}If $E$ is
non-null Borel, then $f(x):=|E\cap (E+x)|$ is continuous at $x=0$,
and so for some $\varepsilon =\varepsilon (E)>0$
\[
E\cap (E+x)\text{ is non-null, for }|x|<\varepsilon .
\]
More generally, $f(x_{1},\dots,x_{p}):=|(E+x_{1})\cap \cdots\cap
(E+x_{p})|$ is continuous at $x=(0,\dots,0)$, and so for
some $\varepsilon =\varepsilon _{p}(E)>0$
\[
(E+x_{1})\cap \cdots\cap (E+x_{p})\text{ is non-null, for }
|x_{i}|<\varepsilon\quad(i=1,\dots,p).
\]
\end{Kemp}

\begin{proof}[Proof 1](After \cite{BHW}\index{Bergelson,
V.}\index{Hindman, N.}\index{Weiss, B.}; cf.\ e.g.\ \cite[Th.\ 6.2 and
7.5]{BOst6}.)  Let $t$ be a density point\index{density point|(} of
$E$. Choose $\varepsilon >0$ such that
\[
|E\cap B_{\varepsilon }(t)|>\frac{3}{4}|B_{\varepsilon }(0)|.
\]
Now $|B_{\varepsilon }(t)\backslash B_{\varepsilon }(t+x)|\leq
(1/4)|B_{\varepsilon }(t+x)|$ for $x\in B_{\varepsilon /2}(0)$, so
\[
|E\cap B_{\varepsilon }(t+x)|>\frac{1}{2}|B_{\varepsilon }(0))|.
\]
By invariance of Lebesgue measure we have
\[
|(E+x)\cap B_{\varepsilon }(t+x)|>\frac{3}{4}|B_{\varepsilon }(0))|.
\]
But, again by invariance, as $B_{\varepsilon }(t)+x=B_{\varepsilon }(0)+t+x$
this set has measure $|B_{\varepsilon }(0)|$. Using $|A_{1}\cup
A_{2}|=|A_{1}|+|A_{2}|-|A_{1}\cap A_{2}|$ with $A_{1}:=E\cap
B_{\varepsilon}(t+x)$
and $A_{2}:=(E+x)\cap B_{\varepsilon }(t+x)$ now yields
\[
|E\cap (E+x)|\geq |E\cap (E+x)\cap (B_{\varepsilon }(t)+x)|>\frac{5}{4}
|B_{\varepsilon }(0)|-|B_{\varepsilon }(0)|>0.
\]
Hence, for $x\in B_{\varepsilon /2}(0)$, we have $|E\cap (E+x)|>0$.

For the $p$-fold form we need some notation. Let $t$ again denote a
\Index{density point} of $E$ and $x=(x_{1},\dots,x_{n})$ a vector of
variables. Set
$A_{j}:=B(t)\cap E\cap (E+x_{j})$ for $1\leq j\leq n$. For each multi-index
$\mathbf{i}=(i(1),\dots,i(d))$ with $0<d<n$, put
\begin{align*}
f_{\mathbf{i}}(x) &:=|A_{i(1)}\cap \cdots\cap A_{i(d)}|; \\
f_{n}(x) &:=|A_{1}\cap \cdots\cap A_{n}|,\qquad f_{0}=|B(t)\cap E|.
\end{align*}

We have already shown that the functions $f_{j}(x)=|B(t)\cap E\cap
(E+x_{j})| $ are continuous at $0$. Now argue inductively: suppose that, for
$\mathbf{i} $ of length less than $n$, the functions $f_{\mathbf{i}}$ are
continuous at $(0,\dots,0)$. Then for given $\varepsilon >0$, there exists $\delta
>0$ such that for $\norm{x}<\delta $ and each such index $\mathbf{i}$ we have
\[
-\varepsilon <f_{\mathbf{i}}(x)-f_{0}<\varepsilon ,
\]
where $f_{0}=|B(t)\cap E|$. Noting that
\[
\bigcup_{i=1}^{n}A_{i}\subset B(t)\cap E,
\]
and using upper or lower approximations, according to the signs in the
\Index{inclusion-exclusion} identity
\[
\left|\bigcup_{i=1}^{n}A_{i}\right|=\sum_{i}|A_{i}|
-\sum_{i<j}|A_{i}\cap A_{j}|+\cdots+(-1)^{n-1}\left|\bigcap_{i}A_{i}\right|,
\]
one may compute linear functions $L(\varepsilon )$, $R(\varepsilon )$ such that
\[
L(\varepsilon )<f_{n}(x)-f_{0}<R(\varepsilon ).
\]
Indeed, taking $x_{i}=0$ in the identity, both sides collapse to the value
$f_{0}$. Continuity follows.
\end{proof}

\begin{proof}[Proof 2]Apply instead Theorem 61.A of
\cite[Ch.\ XII, p.\ 266]{Hal}\index{Halmos, P. R.}
to establish the base case, and then proceed inductively as before.
\end{proof}

\begin{Cor}Theorem K holds for non-meagre Baire sets $E$
in place of Borel sets in the form:

for each $p$ in $\mathbb{N}$ there exists
$\varepsilon =\varepsilon _{p}(E)>0$ such that
\[
(E+x_{1})\cap \cdots\cap (E+x_{p})\text{ is non-meagre, for }
|x_{i}|<\varepsilon\quad (i=1,\dots,p).
\]
\end{Cor}

\begin{proof}A non-meagre Baire set differs from an open set by a meagre
set.
\end{proof}

We will now prove Theorem \ref{t:13} using the Generic Completeness Principle
(Theorem \ref{t:4}); this amounts to proceeding in two steps. To motivate the
proof strategy, note that the embedding property is upward-\Index{hereditary}
(i.e.\ monotonic in the sense of Section 3): if $T$ includes a subsequence of
$z_{n} $ by a shift $t$ in $T$, then so does any superset of $T$. We first
consider a non-meagre $\mathcal{G}_{\delta }$/non-null closed set $T$, just as
in \cite{BHW}, modified to admit the consecutive format. We next deduce the
theorem by appeal to $\mathcal{G}_{\delta }$
inner-regularity\index{Gdelta@$G_\delta$!inner regularity} of
category/measure and Generic Dichotomy. (The subset $E$ of exceptional shifts
can only be meagre/null.)

\begin{theorem}[\textbf{Generalized BHW Lemma---Existence of se-\break
quence embedding}; cf.\ {\cite[Lemma 2.2]{BHW}}]\label{t:14}\index{Bergelson,
V.}\index{Hindman, N.}\index{Weiss, B.}For $T$ Baire
non-meagre\break /measurable non-null and a null sequence $z_{n}\rightarrow 0$,
there exist $t\in T$ and an infinite $\mathbb{M}_{t}$ such that
\[
\{t+\bar{z}_{pm}:m\in \mathbb{M}_{t}\}\subseteq T.
\]
\end{theorem}

\begin{proof}The conclusion of the theorem is inherited by supersets (is
upward hereditary), so without loss of generality we may assume that $T$ is
Baire non-meagre/measurable non-null and completely
metrizable\index{metrizability!complete metrizability}, say under a metric
$\rho =\rho _{T}$. (For $T$ measurable non-null we may pass down to a compact
non-null subset, and for $T$ Baire non-meagre we simply take away a meagre set
to leave a Baire non-meagre $\mathcal{G}_{\delta }$ subset.) Since this is
an \Index{equivalent metric}, for each $a\in T$ and $\varepsilon >0$ there
exists $\delta =\delta (\varepsilon )>0$ such that $B_{\delta }(a)\subseteq
B_{\varepsilon }^{\rho }(a)$. Thus, by taking $\varepsilon =2^{-n-1}$ the
$\delta $-ball $B_{\delta }(a)$ has $\rho $-diameter less than $2^{-n}$.

Working inductively in steps of length $p$, we define subsets of $T$ (of
possible translators) $B_{pm+i}$ of $\rho $-diameter less than $2^{-m}$ for
$i=1$, \dots, $p$ as follows. With $m=0$, we take $B_{0}=T$. Given $n=pm$ and
$B_{n} $ open in $T$, choose $N$ such that $|z_{k}|<\min \{\frac{1}{2}
|x_{n}|,\varepsilon _{p}(B_{n})\}$, for all $k>N$. For $i=1$, \dots, $p$, let
$x_{n-1+i}=z_{N+i}\in Z$; then by Theorem K or its Corollary $B_{n}\cap
(B_{n}-x_{n})\cap \cdots\cap (B_{n}-x_{n+p})$ is non-empty (and open). We may
now choose a non-empty subset $B_{n+i}$ of $T$ which is open in $A$ with
$\rho $-diameter less than $2^{-m-1}$ such that $\cl_{T}B_{n+i}\subset
B_{n}\cap (B_{n}-x_{n})\cap \cdots\cap (B_{n}-x_{n+i})\subseteq B_{n+i-1}$. By
completeness, the intersection $\bigcap_{n\in \mathbb{N}}B_{n}$ is
non-empty. Let
\[
t\in \bigcap_{n\in \mathbb{N}}B_{n}\subset T.
\]
Now $t+x_{n}\in B_{n}\subset T$, as $t\in B_{n+1}$, for each $n$. Hence
$\mathbb{M}_{t}:=\{m:z_{mp+1}=x_{n}$ for some $n\in \mathbb{N}\}$ is
infinite. Moreover, if $z_{pm+1}=x_{n}$ then
$z_{pm+2}=x_{n+1}$, \dots, $z_{pm+p}=x_{n+p-1}$ and so
\[
\{t+\bar{z}_{pm}:m\in \mathbb{M}_{t}\}\subseteq T.\qedhere
\]
\end{proof}

We now apply Theorem \ref{t:3} (Generic Dichotomy) to extend Theorem
\ref{t:14} from an existence to a genericity statement, thus completing the
proof of Theorem \ref{t:13}.

\begin{theorem}[\textbf{Genericity of sequence embedding}]\label{t:15}For $T$
Baire/ measurable and $z_{n}\rightarrow 0$, for generically
all $t\in T$ there exists an infinite $\mathbb{M}_{t}$ such that
\[
\{t+\bar{z}_{pm}:m\in \mathbb{M}_{t}\}\subseteq T.
\]
Hence, if $Z\subseteq X$ accumulates at $0$ (has
an accumulation point there), then for some $t\in T$ the set
$Z\cap (T-t)$ accumulates at $0$ (along $Z$).
Such a $t$ may be found in any open set with which $T$
has non-null intersection.
\end{theorem}

\begin{proof}Working as usual in $X$, the correspondence
\[
F(T):=\bigcap_{n\in \omega }\bigcup_{m>n}[(T-z_{pm+1})\cap
\cdots\cap (T-z_{pm+p})]
\]
takes Baire sets to Baire sets and is monotonic. Here $t\in F(T)$ iff there
exists an infinite $\mathbb{M}_{t}$ such that$\{t+\bar{z} _{pm}:m\in
\mathbb{M}_{t}\}\subseteq T$. By Theorem \ref{t:14} $F(T)\cap T\neq
\varnothing $, for $T$ Baire non-meagre, so we may appeal to Generic Dichotomy
(Th.\ \ref{t:3}) to deduce that $F(T)\cap T$ is quasi all of $T$ (cf.\ Example
1 of Section 3).

With the main assertion proved, let $Z\subseteq X$ accumulate at
$0$ and suppose that $z_{n}$ in $Z$ converges to $0$. Take $p=1$.
Then, for some $t\in T$, there is an infinite $\mathbb{M}_{t}$ such that
$\{t+z_{m}:n\in \mathbb{M}_{t}\}\subseteq T$. Thus $\{z_{m}:n\in \mathbb{M}
_{t}\}\subseteq Z\cap (T-t)$ has $0$ as a joint accumulation point.
\end{proof}

The preceding argument identifies only that $Z\cap (T-t)$ has a point of
simple, rather than essential, contiguity. More in fact is true, as we show
in Theorem \ref{t:16} below.

\begin{nota}Omitting the superscript if context allows, denote by
$\mathcal{M}_{0}^{Ba}$ resp.\ $\mathcal{M}_{0}^{Leb}$ the family of
Baire/Lebesgue-measurable sets which \Index{accumulate essentially} at $0$.
\end{nota}

The following is a strengthened version of the two results in Lemma 2.4 (a)
and (b) of \cite{BHW}\index{Bergelson, V.}\index{Hindman, N.}\index{Weiss, B.} (embraced by (iii) below).

\begin{theorem}[\textbf{Shifted-filter property of
$\mathcal{M}_0$}]\label{t:16}Let $A$
be Baire non-meagre/measurable non-null, $B\in\mathcal{M}_{0}^{Ba/Leb}$.
\begin{enumerate}[{\rm (i)}]
\item If $(A-t)\cap B$ accumulates (simply) at $0$,
then $(A-t)\cap B\in \mathcal{M}_{0}$.
\item For $A$, $B\in \mathcal{M}_{0}$, and generically all
$t\in A$, the set $(A-t)\cap B\in \mathcal{M}_{0}$.
\item For $B\in \mathcal{M}_{0}$ and $t$ such
that $(B-t)\cap B$ accumulates (simply) at $0$, the set
$(B-t)\cap B$ accumulates essentially at $0$.
\end{enumerate}
\end{theorem}

\begin{proof}We will prove (i) separately for the two cases (a) Baire (b)
measure. From KBD (i) implies (ii), while (i) specializes to (iii) by taking
$A=B$.

\begin{enumerate}[(a)]
\item Baire case. Assume that $A$ is Baire non-meagre and that $B$ accumulates
essentially at $0$.

Suppose that $A\cup N_{1}=U\backslash N_{0}$ with $U$ open, non-empty, and
$N_{0}$ and $N_{1}$ meagre. Put $M=N_{0}\cup N_{1}$ and fix $t\in A\backslash
M$, so that $t$ is quasi-any point in $A$; put $M_{t}^{-}:=M\cup
(M-t)$, which is meagre. As $U\backslash M\subset A$, note that by
translation $(U-t)\backslash (M-t)\subset A-t$.

Let $\varepsilon >0$. Without loss of generality $B_{\varepsilon }(0)\subset
U-t$. By the assumption on $B$, $B\cap B_{\varepsilon }(0)$ is non-meagre,
and thus so is $[B\cap B_{\varepsilon }(0)]\backslash M_{t}^{-}$. But the
latter set is included in $B\cap (A-t)$; indeed
\[
\lbrack B\cap B_{\varepsilon }(0)]\backslash M_{t}^{-}\subset \lbrack B\cap
(u-t)]\backslash (M-t)=B_{\varepsilon }(0)\cap B\cap (A-t).
\]
As $\varepsilon $ was arbitrary, $B\cap (A-t)$ accumulates essentially at $0$.

\item Measure case. Let $A$, $B$ be non-null Borel, with $B$ accumulating
essentially at $e$. Without loss of generality both are density-open (all
points are density points\index{density point|)}). By KBD, $(A-t)\cap B$ accumulates (simply) at $e$
for almost all $t\in A$. Fix such a $t$.

Let $\varepsilon >0$ be given. Pick $x\in (A-t)\cap B$ with $|x|<\varepsilon
/2$ (possible since $B\cap (A-t)$ accumulates at $e$). As $x$ and $x-t$ are
density points of $B$ and $A$ (resp.)\ pick $\delta <\varepsilon /2$ such that
\[
|B\cap B_{\delta }(x)|>\frac{3}{4}|B_{\delta }(x)|=\frac{3}{4}|B_{\delta}(e)|
\]
and
\[
|A\cap B_{\delta }(x-t)|>\frac{3}{4}|B_{\delta }(x-t)|=\frac{3}{4}|B_{\delta}(e)|,
\]
which is equivalent to
\[
|(A-t)\cap B_{\delta }(x)|>\frac{3}{4}|B_{\delta }(e)|.
\]
By \Index{inclusion-exclusion} as before, with $A_{1}:=(A-t)\cap B_{\delta }(x)$ and
$A_{2}:=B\cap B_{\delta }(x)$,
\[
|(A-t)\cap B\cap B_{\delta }(x)|>\frac{3}{2}|B_{\delta }(e)|-|B_{\delta}(e)|>0.
\]
But $|x|<\varepsilon /2<\varepsilon -\delta $ so $|x|+\delta <\varepsilon $,
and thus $B_{\delta }(x)\subseteq B_{\varepsilon }(e)$, hence
\[
|(A-t)\cap B\cap B_{\varepsilon }(e)|>0.
\]
As $\varepsilon >0$ was arbitrary, $(A-t)\cap B$ is measurably large at $e$.
\qedhere \end{enumerate}
\end{proof}

\section{Applications: additive combinatorics}\index{additive combinatorics|(}
Recall van der Waerden's theorem\index{Waerden, B. L. van der|(} \cite{vdW} of
1927, that in any finite colouring of the natural numbers, one colour contains
arbitrarily long arithmetic progressions\index{arithmetic progression|(}. This
is one of
Khinchin's\index{Khinchin, A. Ya. [Khintchine, A. Y.]} three pearls of \Index{number
theory} \cite[Ch.\ 1]{Kh}. It has had enormous impact, for instance in Ramsey
theory\index{Ramsey, F. P.}  (\cite{Ram1}; \cite{GRS}, \cite[Ch.\ 18]{HS}) and
additive combinatorics (\cite{TV}; \cite[Ch.\ 14]{HS}).

An earlier theorem of the same type, but for finite partitions of the
\textit{reals} into measurable cells, is immediately implied by the theorem
of Ruziewicz \cite{Ruz}\index{Ruziewicz, St.|(} in 1925, quoted below. We deduce
its category and measure forms from the consecutive form of the KBD
Theorem. The Baire case is new.

\begin{Ruz}Given $p$
positive real numbers $k_{1}$, \dots, $k_{p}$ and any Baire
non-meagre/measurable non-null set $T$, there exist $d$
and points $x_{0}<x_{1}<\cdots<x_{p}$ in $T$ such that
\[
x_{i}-x_{i-1}=k_{i}d,\qquad i=1,\dots,p.
\]
\end{Ruz}

\begin{proof}Given $k_{1}$, \dots, $k_{p}$, define a null sequence by the
condition $z_{pm+i}=(k_{1}+\cdots+k_{i})2^{-m}$ ($i=1$, \dots, $p$). Then
there are $t\in T$ and $m$ such that
\[
\{t+z_{mp+1},\dots,t+z_{mp+p}\}\subseteq T.
\]
Taking $d=2^{-m}$, $x_{0}=t$ and for $i=1$, \dots, $p$
\[
x_{i}=t+z_{mp+i}=t+(k_{1}+\cdots+k_{i})d,
\]
we have $x_{0}<x_{1}<\cdots<x_{p}$ and
\[
x_{i+1}-x_{i}=k_{i}d.\qedhere
\]
\end{proof}

\begin{rems}
\begin{enumerate}[1.]
\item If each $k_{i}=1$ above, then the sequence
$x_{0}$, \dots, $x_{p}$ is an arithmetic progression\index{arithmetic
progression|)} of arbitrarily small step $d$ (which we can take as $2^{-m}$
with $m$ arbitrarily large) and arbitrarily large length $p$. So if
$\mathbb{R}$ is partitioned into a finite number of Baire/measurable cells,
one cell $T$ is necessarily non-meagre/non-null, and contains arbitrarily
long arithmetic progressions of arbitrarily short step. This is similar to the
van der Waerden\index{Waerden, B. L. van der|)} theorem.

\item By referring to the continuity properties of the functions $f_{i}$ in
Theorem K, Kemperman\index{Kemperman, J. H. B.} strengthens the
Ruziewicz\index{Ruziewicz, St.|)} result in the measure case, by establishing
the existence of an upper bound for $d$, which depends on $p$ and $T$ only.
\end{enumerate}
\end{rems}

We now use almost completeness and the \Index{shifted-filter property}
(Th.\ \ref{t:14}) to prove the following.

\begin{BHW}\index{Bergelson, V.}\index{Hindman, N.|(}\index{Weiss, B.}For a
Baire/measurable set $A$ which accumulates essentially\index{accumulate
essentially} at $0$, there exists in $A$ a sequence of reals $\{t_{n}\}$ such
that $\sigma _{F}(t):=\sum_{i\in F}t_{i}\in A$ for every $F\subseteq \omega $.
\end{BHW}

\begin{proof}As in Theorem \ref{t:14} the conclusion is upward \Index{hereditary},
so with\-out loss of generality we may assume that $A$
is
\index{metrizability!complete metrizability}completely metrizable (for $A$
measurable non-null we may pass down to a compact non-null subset accumulating
essentially at 0, and for $A$ Baire non-meagre we simply take away a
meagre\index{meagre|)} set to leave a Baire non-null $\mathcal{G}_{\delta }$
subset).  Let $\rho =\rho _{A}$ be a complete metric equivalent to the
Euclidean metric\index{equivalent metric}. Denote by $\rho\text{-diam}$ the
$\rho $ diameter of a set.

Referring to the \Index{shifted-filter property} of $\mathcal{M}_{e}^{Ba}$ or
$\mathcal{M}_{e}^{Leb}$, we inductively choose decreasing sets
$A_{n}\subseteq A$ and points $t_{n}\in A_{n}$. Assume inductively that:
\begin{enumerate}[(i)]
\item $(A_{n}-t_{n})$ accumulates at $0$,
\item $\sigma _{F}=\sum_{i\in F}t_{i}\in A_{\max F}$, for any finite
set of indices $F\subseteq \{0,1,\dots,\allowbreak n\}$,
\item $\rho\text{-diam}(\sigma _{F})\leq 2^{-n}$ for all finite $F\subseteq
\{0,1,\dots,n\}$.
\end{enumerate}
\noindent By Theorem \ref{t:16},
\[
A_{n+1}:=A_{n}\cap (A_{n}-t_{n})\text{ accumulates essentially at }0.
\]\index{accumulate essentially}
Let $\delta _{n}\in (0,t_{n})$ be arbitrary (to be chosen later). By above,
we may pick
\[
t_{n+1}\in A_{n+1}\cap (0,\delta _{n}/2)\text{ such that }(A_{n+1}-t_{n+1})
\text{ accumulates at }0.
\]
Thus $t_{n}$ is chosen inductively with $t_{n+1}\in A_{n+1}\cap
(A_{n+1}-t_{n+1})$ and $\sum_{i\in I}t_{i}$ convergent for any $I$.
Also
\[
\sum_{i=n+1}^{\infty }t_{i}\leq
t_{n+1}\sum_{i=n+1}^{\infty }2^{-i}=\delta_{n}2^{-n}<\delta_{n}.
\]

Evidently $t_{1}\in A_{1}$. As $A_{n}\subset A_{n+1}\subset A_{n}-t_{n}$, we
see that, as $t_{1}+\cdots+t_{n}\in A_{n}$, we have $t_{1}+\cdots+t_{n+1}\in
A_{n+1}$. More generally, $\sigma_{F}=\sum_{i\in F}t_{i}\in
A_{\max F}$ for any finite set of indices $F\subseteq \{0,1,\dots,n+1\}$. For
$\varepsilon =2^{-n-1}$ there exists $\delta =\delta (\varepsilon )>0$ small
enough such that for all finite $F\subseteq \{0,1,\dots,n+1\}$
\[
B_{\delta }(\sigma _{F})\subseteq B_{\varepsilon }^{\rho }(\sigma _{F}).
\]
Taking $\delta _{n}<\delta (2^{-n-1})$ in the inductive step above implies
that, for any infinite set $I$, the sequence $\sigma_{I\cap \{0,\dots,n\}}$
is Cauchy under $\rho $, and so $\sigma _{I}\in A$.
\end{proof}\index{additive combinatorics|)}

\begin{rem}[\textbf{Generalizations}]Much of the material here (which extends
immediately from additive to multiplicative formats) can be taken over to the
more general contexts of $\mathbb{R}^{d}$ and beyond---to normed
groups\index{group!normed group} (including Banach spaces)\index{Banach,
S.!Banach space}, for which see \cite{BOst6}. We choose to restrict here to
the line---Kingman's setting---for simplicity, and in view of Mark
Kac's\index{Kac, M.} dictum: No theory can be better than its best
example.\index{measure theory|)}\index{Baire, R.-L.!Baire
set|)}\index{topology!bitopology|)}
\end{rem}

\subparagraph{Postscript}It is no surprise that putting a really good theorem
and a really good mathematician together may lead to far-reaching
consequences. We hope that John Kingman\index{Kingman, J. F. C.!influence}
will enjoy seeing his early work on category still influential forty-five
years later. The link with combinatorics is much more recent, and still
pleases and surprises us---as we hope it will him, and our readers.

\subparagraph{Acknowledgments}It is a pleasure to thank the referee for a close
reading of the paper, Roy Davies\index{Davies, R. O.} for his helpful insights
(Th.\ \ref{t:10}), and Dona Strauss\index{Strauss, D.} for discussions on the
corpus of combinatorial work established by her, Neil Hindman\index{Hindman,
N.|)} and their collaborators. We salute her 75th birthday.

\end{document}